\documentclass[12pt]{amsart}
 \usepackage[dvips]{epsfig}
 \usepackage{enumerate}
 \usepackage{amsgen, amstext,amsbsy,amsopn, amsthm, amsfonts,amssymb,amscd,amsmat
 h,euscript,enumerate,url,verbatim,calc,xypic, mathtools}

 \usepackage{latexsym}
 \usepackage{graphics}
 \usepackage{color}

\newcommand{\proset}{\,\mathrel{\lower 4pt\hbox{$\scriptscriptstyle/$}
\mkern -14mu\subseteq }\,} 
 
 \newtheorem{theorem}{Theorem}[section]
  \newtheorem{corollary}[theorem]{Corollary}
 \newtheorem{lemma}[theorem]{Lemma}

\newtheorem{remark}[theorem]{Remark}
 
 \newtheorem{definition}[theorem]{Definition}
 
 \newtheorem{example}[theorem]{Example}
 \usepackage[backend=biber,style=alphabetic,sorting=ynt]{biblatex}

\numberwithin{equation}{section}
\bibliography{mybib.bib}
\def\ker{\operatorname{ker}}
\def\coker{\operatorname{coker}}

\usepackage{amsmath}
\usepackage{tikz-cd}
 \makeatother

\date{\today}

\title{Homotopy Invariance of $K$-groups using Grayson's Technique}
 \author{Sourayan Banerjee} 
 
 \address{Department of Mathematics, Indian Institute of Technology,  Kanpur\\Uttar Pradesh-208016, India}
 \email{sourayanb@iitk.ac.in }
 \keywords{Nil K-groups, Binary complexes}
 \subjclass[2010]{Primary 19D06; Secondary 19D35, 18E10.}
 \thanks{The author was supported by the Institute Post-Doctoral fellowship, Indian Institute of Technology, Kanpur}
\begin{document}
 \maketitle

 \begin{abstract}
     Homotopy invariance of $K$-theory has always been a point of interest. In this article, with the help of the generators of Nil$K$-groups using Grayson's technique, it is shown that if $R$ is a Pr\"{u}fer domain, then $K_n(R) \cong K_n(R[s])$ for all $n>0.$ This is a specific case of the already published work of the author and Vivek Sadhu. However, contrary to the method used before in \cite{BS22}, we specifically prove the isomorphism by showing that Nil$K$-groups vanish.
 \end{abstract}
 \section{Introduction}
Quillen gave the definition of an exact category and also defined the higher $K$-groups over a small exact category $\mathcal N$. Classically, he proved that if $R$ is a regular Noetherian ring, then $K_n(R[t_1,t_2,.....,t_m]) \cong K_n(R)$, for every $n \in \mathbb Z$ and $m \in \mathbb N$. In 2012, Grayson defined higher $K$-groups over exact categories (that support long exact sequences) as a quotient of the Grothendieck group of the category of bounded acyclic binary complexes over $\mathcal N$, which we denote as $B^{q^{n}}(\mathcal N)$. To be precise, he proved that $$K_n^{Gr}(\mathcal N) \cong K_0(B^{q^{n}}(\mathcal N))/T^n_{\mathcal N} \cong K_n^Q(\mathcal N),$$ where $K_n^{Gr}$ and $K_n^Q$ denote the $K$-groups (for $n\geq 0$) obtained via the definition of Grayson and Quillen respectively, and $T^n_\mathcal N$ denotes the subgroup of $K_0(B^{q^{n}}(\mathcal N))$ generated by the diagonal complexes. Throughout the article we would assume that $K_n(\mathcal N):= K_n^{Gr}(\mathcal N)$. Given a commutative unital ring $R,$ let ${\bf P}(R)$  denote the category of finitely generated projective $R$-modules. Let ${\bf Nil}(R)$ be a category consisting of all pairs $(P, \nu),$ where $P$ is a finitely generated $R$-module and $\nu$ is a nilpotent endomorphism of $P.$ Let ${\rm Nil}_{0}(R)$ denote the kernel of the forgetful map $K_{0}({\bf Nil}(R)) \to K_{0}({\bf P}(R)) =: K_{0}(R).$ The group ${\rm Nil}_{0}(R)$ is generated by elements of the form $[(R^{n}, \nu)]- [(R^{n}, 0)]$ for some $n$ and some nilpotent endomorphism $\nu.$  Let ${\rm Nil}(R)$ denote the homotopy fibre of forgetful functor $K{\bf Nil}(R) \to  K{\bf P}(R):= K(R).$ The $n$-th Nil group ${\rm Nil}_{n}(R)$ is $\pi_{n}{\rm Nil}(R).$ Since the forgetful functor splits,  $K{\bf Nil}(R) \simeq {\rm Nil}(R) \times K(R).$  This implies that $K_{n}{\bf Nil}(R)\cong {\rm Nil}_{n}(R)\bigoplus K_{n}(R)$ for every ring $R.$  There is an isomorphism  ${\rm Nil}_{n}(R)\cong NK_{n+1}(R)= \ker [K_{n+1}(R[t]) \stackrel{t\mapsto 0}\to K_{n+1}(R)]$ (see Theorem V.8.1 of \cite{wei1}). Since $K_{n+1}(R[t]) \cong K_{n+1}(R) \oplus NK_{n+1}(R)$ because of the canonical splitting, the group $NK_{n+1}(R)$ or Nil$_n(R)$ being trivial implies that $K_{n+1}(R[t]) \cong K_{n+1}(R)$. Thus, to show $NK_{n+1}(R)$ is trivial, we show that the generators of ${\rm Nil}_n(R)$ found in \cite{BS24} are trivial whenever the ring is semihereditary. What we precisely prove is the following:
\begin{theorem}
    Let $R$ be a Pr\"{u}fer domain, then the base change  $K_n(R[t]) \rightarrow K_n(R)$ is an isomorphism for all $n \geq 0$. 
\end{theorem}
As a corollary, we get that 
\begin{corollary}
    If $R$ is a Valuation ring or Dedekind domain, then $K_n(R[t]) \cong K_n(R).$
\end{corollary}
Both the results have already been proved in far generality, i.e., $K$-theory is homotopy invariant for both Pr\"{u}fer domains and Valuation rings. But the motivation behind this article is to show an application of the already found generators of $\rm Nil$-$K$-groups. For instance, over regular noetherian rings to show homotopy invariance of $K$-theory, it is sufficient to show that $K_n(R) \cong K_n(R[s])$ along with $K_n(R[s,s^{-1}]) \cong K_n(R) \oplus K_{n-1}(R).$ Thus, using the generators in this article, we would see that the following fundamental theorem holds.
\begin{theorem}{\label{fdthm}}
    If $R$ is a Dedekind domain, then $K_n(R[s]) \cong K_n(R)$, in addition $$K_n(R[s,s^{-1}]) \cong K_n(R) \oplus K_{n-1}(R).$$ 
\end{theorem}
This article will be summarized as follows. Preliminaries regarding Pr\"{u}fer domains will be discussed in Section 2. Especially, we will mention that torsion-free finitely generated modules over a Pr\"{u}fer domain are always projective, which will play a central role in proving that the base change map is homotopy invariant. In section 3, we will discuss Grayson's technique in detail. To do that, we will first recall the definition of acyclic binary complexes over a small exact category $\mathcal N$ that supports long exact sequences, and then describe $K_n(\mathcal N)$ with the help of generators as isomorphism classes of $n$-dimensional binary acyclic complexes. Section 4 would be dedicated to recalling the generators of $\rm{Nil}$-$K$ groups with the help of Grayson's technique.  Finally, in Section 5, we prove our theorem and lay some groundwork for future direction. 
 \section{Basics of Pr\"ufer Domains}
 In this section, we will recall what a Pr\"{u}fer domain is, alongside some of its useful and relevant properties.
 Let $R$ be a commutative ring. A finitely generated $R$-module $M$ is called {\it coherent} if every finitely generated submodule of $M$ is finitely presented. The ring $R$ is {\it coherent} if it is a coherent module over itself, i.e., every finitely generated ideal of $R$ is finitely presented.
\begin{definition}
An integral domain $R$ is a {\it Pr\"{u}fer} domain if every finitely generated nonzero ideal $I$ of $R$ is invertible. Equivalently, $R_{\mathfrak{p}}$ is a valuation domain for all prime ideals $\mathfrak{p}$ of $R.$ 
\end{definition}
\begin{example}{\rm
\begin{enumerate}
 \item Every valuation domain is Pr\"{u}fer.
 
 \item The ring of integer-valued polynomials $${\rm Int}(\mathbb{Z})= \{f\in \mathbb{Q}[x]| f(\mathbb{Z})\subset \mathbb{Z}\}$$ is a Pr\"{u}fer domain. The ideal $\{f\in {\rm Int}(\mathbb{Z})| f(0) ~{\rm is~  even}\}$ is not finitely generated. Thus, ${\rm Int}(\mathbb{Z})$ is a non-noetherian Pr\"{u}fer domain. More generally, if $D$ is a Dedekind domain with finite residue fields then ${\rm Int}(D)$ is a Pr\"{u}fer domain. For example, one can consider  $D=\mathcal{O}_{K},$ the ring of integers of a number field $K$ (see \cite{cahen}).

  \item A domain whose finitely generated ideals are principal is called a Bezout domain. Every such domain is a Pr\"{u}fer domain. As an example, we can look at the ring of algebraic integers $\mathbb{A}$ that has a minimal polynomial in $\mathbb{Z}[x]$. Every finitely generated ideal of $\mathbb A$ is principal.

\item  A commutative ring with unity $R$ is called semi-hereditary if every finitely generated ideal of $R$ is a projective $R$-module. Every semi-hereditary domain is a Pr\"{u}fer domain.
\end{enumerate}

 }
\end{example}

We now state some well known facts pertaining to Pr\"{u}fer domains.

\begin{enumerate}
 \item If $R$ is a Pr\"{u}fer domain then it is integrally closed.
 
 \item Every torsion-free module over a Pr\"{u}fer domain is flat.

\item  Every finitely generated torsion-free module over a Pr\"{u}fer domain is projective.
 
 \item A Pr\"{u}fer domain is always  coherent.
 
 \item If $R$ is a Pr\"{u}fer domain then the weak global dimension of $R$ is at most one (see p. 25 of \cite{Glaz}).

 \end{enumerate}
As we have finished briefly discussing the preliminaries of Pr\"{u}fer domains, now in the next section we will recall the Grayson's technique to define the higher $K$-groups.

 \section{Binary Complexes and Grayson's Theorem}
 Let $\mathcal{N}$ denote an exact category. A bounded chain complex $N$ in $\mathcal{N}$ is said to be an {\it acylic chain complex} if each differential $d_{i}: N_{i} \to N_{i-1}$ can be factored as $N_{i} \to Z_{i-1} \to N_{i-1}$ such that each $0\to Z_{i} \to N_{i} \to Z_{i-1}\to 0$ is a short exact sequence of $\mathcal{N}$ (see Definition 1.1 of \cite{Gray}).\\
 \begin{remark}
     One thing to notice here is that all bounded acyclic chain complexes are exact sequences, but not all bounded exact sequences are acyclic.\\
     For example, fix the exact category $\mathcal N$ to be the category of finitely generated free $R$-modules, $\rm Free (R)$. Let $P \in {\bf P}(R)$ which is not free but stably free, i.e., for some $n$ and $m$ $\in \mathbb N$, $P \oplus R^m \cong R^n.$ Then the following sequence $$\xymatrix{
     R^m \ar[r]& R^n \ar[dr]\ar[r] & R^n \ar[r] &R^m\\
     & & P \ar[u]
     }$$ is exact but not acyclic due to the fact that $P$ is not free.
 \end{remark}Let $C\mathcal{N}$ denote the category of bounded chain complexes in $\mathcal{N}.$ The full subcategory of $C\mathcal{N}$ consisting of  bounded acylic complexes in $C\mathcal{N}$ is denoted by $C^{q}\mathcal{N}$ (see section 2 of \cite{Gray}). The category $C^{q}\mathcal{N}$ is exact.

A chain complex in $\mathcal{N}$ with two differentials (not necessarily distinct) is called a {\it binary chain complex} in $\mathcal{N}.$ In other words, it is a triple $(N_{*}, d, d^{'})$ with $(N_{*}, d)$ and $(N_{*}, d^{'})$ are in $C\mathcal{N}$ (see Definition 3.1 of \cite{Gray}). A morphism between two binary complexes $(N_{*}, d, d^{'})$ and  $(\tilde{N}_{*}, \tilde{d}, \tilde{d}^{'})$ is a morphism between the underlying graded objects $N$ and $\tilde{N}$ that commutes with both differentials. The category of bounded binary complexes in $\mathcal{N}$ is denoted by $B\mathcal{N}.$ There is always a diagonal functor (see Definition 3.1 of \cite{Gray}) $$ \Delta: C\mathcal{N} \to B\mathcal{N}, ~{\rm defined ~by}~ \Delta((N_{*}, d))= (N_{*}, d, d).$$ If a binary complex is in the image of $\Delta$, we call it a diagonal binary complex, i.e. when $d = d'$ for $(N_*,d,d') \in B\mathcal{N}.$ As before, let $B^{q}\mathcal{N}$ denote the full subcategory of $B\mathcal{N}$ whose objects are bounded acylic binary complexes. This is also an exact category (see section 3 of \cite{Gray}).

By iterating, one can define exact category $(B^{q})^{n}\mathcal{N}=B^{q}B^{q}\cdots B^{q}\mathcal{N}$ for each $n\geq 0.$ An object of the exact category  $(B^{q})^{n}\mathcal{N}$ of bounded acylic binary multi-complexes of dimension $n$ in $\mathcal{N}$ is a bounded $\mathbb{Z}^{n}$- graded collection of objects of $\mathcal{N},$ together with a pair of acyclic differentials $d^{i}$ and $\tilde{d^{i}}$ in each direction $1\leq i\leq n,$ where the differentials $d^{i}$ and $\tilde{d^{i}}$ commute with $d^{j}$ and $\tilde{d^{j}}$ whenever $i\neq j.$ Thus, a typical object looks like $(N_{*}, (d^{1}, \tilde{d^{1}}), (d^{2}, \tilde{d^{2}}), \dots, (d^{n}, \tilde{d^{n}})),$ where $N_{*}$ is a
bounded $\mathbb{Z}^{n}$- graded collection of objects of $\mathcal{N}$ (see Defintion 7.3 of \cite{Gray}). We say that an acyclic binary multi-complex $(N_{*}, (d^{1}, \tilde{d^{1}}), (d^{2}, \tilde{d^{2}}), \dots, (d^{n}, \tilde{d^{n}}))$ is {\it diagonal} if $d^{i}=\tilde{d^{i}}$ for at least one $i$ (see Corollary 7.4 of \cite{Gray} and Definition 2.2). 


 
 In \cite{Nenasheb}, Nenashev described $K_{1}$-group in terms of generator and relations using the notion of double exact sequences. Motivated by \cite{Nenasheb},  Grayson defines higher $K$-groups in terms of generators and relations using binary complexes (see \cite{Gray}). However, Nenashev's  $K_{1}$-group agree with Grayson's $K_{1}$-group (see Corollary 4.2 of \cite{KKW}). In the rest of the paper, we assume the following as the definition of higher $K$-groups.

\begin{definition}\label{main def}(see Corollary 7.4 of \cite{Gray}) Let $\mathcal{N}$ be an exact category. For $n\geq 1,$ $K_{n}\mathcal{N}$ is the abelian group having generators $[N],$ one for each object $N$ of $(B^{q})^{n}\mathcal{N}$ and the relations are:
\begin{enumerate}
 \item $[N^{'}] + [N^{''}]=[N]$ for every short exact sequence $0 \to N^{'} \to N \to N^{''}\to 0$ in $(B^{q})^{n}\mathcal{N};$
 \item $[T]=0$ if $T$ is a diagonal acyclic binary multi-complex.
\end{enumerate}
\end{definition}
The above definition can be interpreted as a quotient group which we record in \\Remark \ref{k grp as a qt} but before that, we define a subgroup that will be used frequently.
\begin{definition}(See Definition 1.39 of \cite{HarrisT})\label{Dc}
   An acyclic binary multi-complex \\ $(N_{*}, (d^{1}, \tilde{d^{1}}), (d^{2}, \tilde{d^{2}}), \dots, (d^{n}, \tilde{d^{n}}))$ over an exact category $\mathcal{N}$ is {\it diagonal} if $d^{i}=\tilde{d^{i}}$ for at least one $i,$ i.e. it is in the image of the  functor $$\Delta^i : C^q(B^{q})^{n-1}\mathcal{N} \rightarrow (B^{q})^{n}\mathcal{N}$$ that duplicates the differential in the $i^{th}$ direction for some $1 \leq i \leq n+1.$  Let $T_{\mathcal{N}}^{n}$ denote the subgroup of $K_{0}(B^{q})^{n}\mathcal{N}$ generated by the classes of the diagonal acyclic binary multi-complexes in $K_{0}(B^{q})^{n}\mathcal{N}.$ 
    
\end{definition}
\begin{remark}\label{k grp as a qt}{\rm
 In Definition \ref{main def}, if we only consider the relation (1) then it is just $K_{0}(B^{q})^{n}\mathcal{N}.$ Then $K_{n}\mathcal{N}\cong K_{0}(B^{q})^{n}\mathcal{N}/T_{\mathcal{N}}^{n}.$}
\end{remark}

\begin{lemma}\label{harris observation}(see Lemma 2.7 of \cite{Harris}) For each $n\geq 1$, there is a split short exact sequence 
$$0 \to K_{n-1}C^{q}\mathcal{N}\stackrel{\Delta}\to  K_{n-1}B^{q}\mathcal{N} \to K_{n}\mathcal{N}\to 0,$$ which is functorial in $\mathcal{N}.$ 
 \end{lemma}
 \begin{proof}
  See Lemma 2.7 of \cite{Harris}.
 \end{proof}
\begin{remark}\label{further observ to harris}{\rm
 For $n=1,$ $ K_{0}C^{q}\mathcal{N}\cong {\rm im}(\Delta)=T_{\mathcal{N}}^{1}.$ Thus, the above lemma implies that $K_{0}B^{q}\mathcal{N}\cong T_{\mathcal{N}}^{1} \oplus K_{1}\mathcal{N}.$ }
 
 \end{remark}
 Before proceeding further, let us fix some notations. Throughout the rest of the article $({B^q})^{n}_{[0,2]}\mathcal N$ will denote binary multi-complexes that are supported on $[0,2]$ in every direction, similarly we denote ${(C^q)}^n\mathcal N.$
 Interestingly, Kasprowski and Winges pointed out that it is sufficient to consider the quotient of $K_0(({B^q})^n_{[0,2]}\mathcal N$ in order to define the higher $K$-groups (see \cite{WK} Theorem 1.3 and Theorem 1.4). In this article, all the generators used in all the computations will be assumed over the category $({B^q})^n_{[0,2]}$. 
\section{\rm{Nil}$K$-groups and their Generators}
Let $R$ be a commutative ring with unity. The category of finitely generated projective $R$-modules is denoted by ${\bf P}(R).$ Let ${\bf Nil}(R)$ denote the category whose objects are pairs $(P, \nu),$ where $P$ is a finitely generated projective $R$-module and $\nu$ is a nilpotent endomorphism of $P.$ A morphism $f: (P_{1}, \nu_{1}) \to (P_{2}, \nu_{2})$ is a $R$-module map $f: P_{1} \to P_{2}$ such that $f\nu_{1}=\nu_{2}f.$ Both ${\bf P}(R)$ and ${\bf Nil}(R)$ are exact categories.\\
There is always a canonical forgetful functor from $\bf Nil$$(R)$ to $\bf{P}$$(R)$ that induces the following map between $K_0({\bf Nil}(R)) \rightarrow K_0(R)$; $[P, \nu] \mapsto [P]$. This is a split surjection and thus, $K_0({\bf Nil}(R)) \cong K_0(R) \oplus \rm Nil_0(R).$ Since $K_n$ is a functor, in general we have $K_n({\bf Nil}(R)) \cong K_n(R) \oplus {\rm{Nil}}_n(R).$\\
Because $NK_n(R) \cong {\rm{Nil}_{n-1}}(R)$, in order to show that $K_n(R) \cong K_n(R[s])$; for $n>0$, it is both necessary and sufficient that we show ${\rm{Nil}}_n(R)$; for all $n>0$ are trivial for $n \geq 0.$ In this section, first we recall the generators of ${\rm{Nil}}_n(R)$ and so 
\begin{theorem}[\cite{BS24}~Theorem 5.6]
    The group ${\rm Nil}_{n}(R)$ is generated by elements of the form \small$$([(F_{*}, (f_{1}, f_{1}^{'}), (f_{2}, f_{2}^{'}), \dots, (f_{n}, f_{n}^{'}), \nu)]- [(F_{*}, (f_{1}, f_{1}^{'}), (f_{2}, f_{2}^{'}), \dots, (f_{n}, f_{n}^{'}), 0)])(~{\rm mod} ~\widetilde{T}_{R}^{n}),$$ \normalsize where $(F_{*}, (f_{1}, f_{1}^{'}), (f_{2}, f_{2}^{'}), \dots, (f_{n}, f_{n}^{'}), \nu)), (F_{*}, (f_{1}, f_{1}^{'}), (f_{2}, f_{2}^{'}), \dots, (f_{n}, f_{n}^{'}), 0))$ are objects of $(B^{q})^{n}{\bf Nil(Free}(R)).$
\end{theorem}
The subgroup $~\widetilde{T}_{R}^{n}$ arises from chasing the following commutative diagram (see \cite{BS24}~Lemma 3.4 and Section 3).
 $$ \begin{CD}
        @.  0 @. 0 @.0  \\
           @. @VVV  @VVV  @VVV  \\
         0 @>>> \widetilde{T}_{R}^{n} @>>> T_{{\bf Nil}(R)}^{n} @>>> T_{{\bf P}(R)}^{n}@>>> 0 \\
            @.  @VVV    @VVV   @VVV \\
        0 @>>> \widetilde{{\rm Nil}}_{n}(R) @>>> K_{0}(B^{q})^{n} {\bf Nil}(R) @>>> K_{0}(B^{q})^{n} {\bf P}(R) @>>> 0 \\
             @. @VVV   @VVV    @VVV  \\
        0 @>>> {\rm Nil}_{n}(R) @>>>K_{n}{\bf Nil}(R) @>>> K_{n}(R) @>>> 0 \\
         @.  @VVV   @VVV  @VVV  \\
        @. 0 @. 0 @. 0,
       \end{CD}$$

To make the reader familiarize with the generators written above we give a brief description of them when $n =0$ and $1$. 
    \begin{enumerate}
        \item $\rm{Nil}_0(R)$ is generated by the class of elements of the form $[R^n,\nu]-[R^n,0].$
        \item $\rm{Nil}_1(R)$ is generated by the class of elements of the form $[F_*,d,d',\nu] - [F_*,d,d',0]$. As stated before, we would be working with binary complexes supported on $[0,2]$. Thus, $(F_*,d,d',\nu) \in B^q_{[0,2]}{\bf Nil(Free}(R))$ has an expression as the following:
\begin{center}
\begin{tikzcd}
  0 \ar[r, yshift=2pt]\ar[r, yshift=-2pt] &R^n \ar[r, "d_2", yshift=2pt]\ar[r, "d'_2"', yshift=-2pt]\ar[d, "\nu_2"] &R^m \ar[r, "d_1", yshift=2pt]\ar[r, "d'_1"', yshift=-2pt]\ar[d, "\nu_1"] &R^k \ar[r, yshift=2pt]\ar[r, yshift=-2pt]\ar[d, "\nu_0"] & 0\\
   0 \ar[r, yshift=2pt]\ar[r, yshift=-2pt] &R^n \ar[r, "d_2", yshift=2pt]\ar[r, "d'_2"', yshift=-2pt] &R^m \ar[r, "d_1", yshift=2pt]\ar[r, "d'_1"', yshift=-2pt] &R^k \ar[r, yshift=2pt]\ar[r, yshift=-2pt] & 0
\end{tikzcd}
\end{center}
where the vertical nilpotent endomorphisms commute with both the differentials.
\item For ${\rm{Nil}}_n(R)$ the above diagram is iterated. 
    \end{enumerate}
In \cite{BS24} we had already proven that ${\rm Nil}_0(R)$ is trivial for $R$ being a Pr\"{u}fer domain. We recall the statement here, since we are going to prove the genralized statement later in Section $5$.
\begin{theorem}[\cite{BS24}~Theorem 3.2]\label{nil0}
     Let $R$ be a commutative ring with unity. Assume that every finitely generated torsion-free $R$-module is projective. Then ${\rm Nil}_{0}(R)=0.$
\end{theorem}
\begin{proof}
    See \cite{BS24}, Remark $3.3$ and the proof of  Lemma $3.1$. 
\end{proof}

\section{Main Results}
Now that we have set the background for our work, in this section,  we generalize Theorem \ref{nil0}. But first, we need to understand the behaviour of the generators of the group ${\rm Nil}_1(R),$ whenever $R$ is a Pr{\"u}fer domain. To prove the triviality of ${\rm Nil_1}(R)$ it suffices to show that $[F_*,d,d',\nu] = [F_*,d,d',0]$. Let $(R^n, \nu)$ be an object of ${\bf Nil}(R),$ and assume that $\nu$ is non-zero nilpotent. Then, there exist a least $m\in \mathbb{N}$ such that $\nu^{m}=0$ and $\nu^{r}\neq 0$ for $r< m.$ Then we have a chain of $R$-modules
\begin{equation}\label{kerem}
0 \subseteq \ker(\nu)\subseteq \ker(\nu^{2}) \subseteq \dots \subseteq \ker(\nu^{n-1})\subseteq \ker(\nu^{n})=R^{n}.
\end{equation}
Since  $\frac{\ker(\nu^{i+1})}{ker(\nu^{i})} $ is a torsion free $R$-module for $1\leq i \leq m-1$, they are projective as well because they are finitely generated (see \cite{BS24}, Lemma 3.1).\\
Now, let us recall the expression of $(F_*,d,d',\nu)$ is of the form 
\begin{equation}
\begin{tikzcd}
  0 \ar[r, yshift=2pt]\ar[r, yshift=-2pt] &R^n \ar[r, "d_2", yshift=2pt]\ar[r, "d'_2"', yshift=-2pt]\ar[d, "\nu_2"] &R^m \ar[r, "d_1", yshift=2pt]\ar[r, "d'_1"', yshift=-2pt]\ar[d, "\nu_1"] &R^k \ar[r, yshift=2pt]\ar[r, yshift=-2pt]\ar[d, "\nu_0"] & 0\\
   0 \ar[r, yshift=2pt]\ar[r, yshift=-2pt] &R^n \ar[r, "d_2", yshift=2pt]\ar[r, "d'_2"', yshift=-2pt] &R^m \ar[r, "d_1", yshift=2pt]\ar[r, "d'_1"', yshift=-2pt] &R^k \ar[r, yshift=2pt]\ar[r, yshift=-2pt] & 0
\end{tikzcd}.
\end{equation}
Thus, using the nilpotent endomorphisms in each column, we get the following split exact sequence of short exact sequences 
\begin{equation}\label{diagram}
    \begin{tikzcd}
  0 \ar[r, yshift=2pt]\ar[r, yshift=-2pt] &{(\ker(\nu_2^{m-1}),\nu_2)} \ar[r, "d_2", yshift=2pt]\ar[r, "d'_2"', yshift=-2pt]\ar[d] &{(\ker(\nu_1^{n-1}),\nu_1)} \ar[r, "d_1", yshift=2pt]\ar[r, "d'_1"', yshift=-2pt]\ar[d] &{(\ker(\nu_0^{k-1}),\nu_0)} \ar[r, yshift=2pt]\ar[r, yshift=-2pt]\ar[d] & 0\\
   0 \ar[r, yshift=2pt]\ar[r, yshift=-2pt] &R^n \ar[r, "d_2", yshift=2pt]\ar[r, "d'_2"', yshift=-2pt]\ar[d] &R^m \ar[r, "d_1", yshift=2pt]\ar[r, "d'_1"', yshift=-2pt]\ar[d] &R^k \ar[r, yshift=2pt]\ar[r, yshift=-2pt]\ar[d] & 0\\
    0 \ar[r, yshift=2pt]\ar[r, yshift=-2pt] &(\frac{\ker(\nu_2^{n})}{\ker(\nu_2^{n-1})}, 0) \ar[r, "d_2", yshift=2pt]\ar[r, "d'_2"', yshift=-2pt] &(\frac{\ker(\nu_1^{m})}{\ker(\nu_1^{m-1})}, 0) \ar[r, "d_1", yshift=2pt]\ar[r, "d'_1"', yshift=-2pt] &(\frac{\ker(\nu_2^{k})}{\ker(\nu_2^{k-1})}, 0) \ar[r, yshift=2pt]\ar[r, yshift=-2pt] & 0.
\end{tikzcd}
\end{equation}
Because nilpotent endomorphisms can have different nilpotent indexes, let us see what would happen if, for instance, the nilpotent index of $\nu_2$ is less than the nilpotent index of $\nu_1.$ Then, in the above diagram, we can choose the minimum of the nilpotent indices among $\nu_2, \nu_1$ and curate the diagram with appropriate adjustment. In general, we would first establish what the ${\rm min}\{r,s,t\}$ (nilpotent indices of \{$\nu_1,\nu_2,\nu_3\}$  respectively) is and then adjust the first row of the above diagram \ref{diagram} accordingly. To put it simply, if $r$ is the minimum, then the first row would be adjusted as follows:
\begin{equation}
    \begin{tikzcd}
        0 \ar[r, yshift=2pt]\ar[r, yshift=-2pt] &{(\ker(\nu_2^{r-1}),\nu_2)} \ar[r, "d_2", yshift=2pt]\ar[r, "d'_2"', yshift=-2pt] &{(\ker(\nu_1^{r-1}),\nu_1)} \ar[r, "d_1", yshift=2pt]\ar[r, "d'_1"', yshift=-2pt] &{(\ker(\nu_0^{r-1}),\nu_0)} \ar[r, yshift=2pt]\ar[r, yshift=-2pt] & 0.
    \end{tikzcd}
\end{equation}
Since, kernel of nilpotents has a nice embedding (see \ref{kerem}), apart from the first row of \ref{diagram} remains unchanged, and it remains a split exact sequence of short exact sequences. Thus by denoting the first row of the diagram as $(\ker^r_*,d,d',\nu)$ and the final row as $(\coker_*,d,d',0)$, we arrive at the following equation in the Grothendieck group, i.e.
\begin{align*}
[F_*, d,d',\nu]&= [\ker_*,d,d',\nu]+[\coker_*,d,d',0]
 \end{align*}
Now in a similar fashion, we can further deduce $$[\ker_*,d,d',\nu] = [\ker^{r-1}_*,d,d',\nu]+[\coker^r_*,d,d',0].$$ Thus, what we find is an iteration of the argument that we had produced as the proof of \cite{BS24}, Theorem 3.2. So we conclude that $[F_*,d,d',\nu] = [F_*,d,d',0]$. This discussion thus implies the following.
\begin{theorem}\label{nil1}
    If $R$ is a Pr\"{u}fer domain then ${\rm{Nil}}_1(R)$is trivial. 
\end{theorem}
Now to show that ${\rm Nil}_n(R)$ is trivial it would suffice to show that  \small$$[(F_{*}, (f_{1}, f_{1}^{'}), (f_{2}, f_{2}^{'}), \dots, (f_{n}, f_{n}^{'}), \nu)]- [(F_{*}, (f_{1}, f_{1}^{'}), (f_{2}, f_{2}^{'}), \dots, (f_{n}, f_{n}^{'}), 0)] \in ~\widetilde{T}_{R}^{n}.$$ The generator $[(F_{*}, (f_{1}, f_{1}^{'}), (f_{2}, f_{2}^{'}), \dots, (f_{n}, f_{n}^{'}), \nu)] \in ({B^q})^n{\bf P}(R)$, so in each direction there is a splitting as we have established in the proof of Theorem \ref{nil1}. All the splittings in each of those directions are compatible with all pairs of morphisms $(f_i,f_{i}^{'})$. For example working with generators in $({B^q})^2{\bf Nil}R$ we generate a split exact sequence in two directions. What we mean by that is, if we work with the generators of the following form, where $R^p:= (R^p, \nu_p)$:
\begin{equation}
\begin{tikzcd}
   0 \ar[r, yshift=2pt]\ar[r, yshift=-2pt] &R^{n'} \ar[r,  yshift=2pt]\ar[r,  yshift=-2pt]\ar[d,  xshift=2pt]\ar[d,  xshift=-2pt] &R^{m'} \ar[r, yshift=2pt]\ar[r,  yshift=-2pt] \ar[d,  xshift=2pt] \ar[d,  xshift=-2pt]&R^{k'} \ar[r, yshift=2pt]\ar[r, yshift=-2pt]\ar[d,  xshift=2pt]\ar[d,  xshift=-2pt] & 0\\
  0 \ar[r, yshift=2pt]\ar[r, yshift=-2pt] &R^n \ar[r,  yshift=2pt]\ar[r,  yshift=-2pt]\ar[d,  xshift=2pt]\ar[d,  xshift=-2pt] &R^m \ar[r, yshift=2pt]\ar[r, yshift=-2pt] \ar[d,  xshift=2pt] \ar[d,  xshift=-2pt]&R^k \ar[r, yshift=2pt]\ar[r, yshift=-2pt]\ar[d,  xshift=2pt]\ar[d,  xshift=-2pt] & 0\\
   0 \ar[r, yshift=2pt]\ar[r, yshift=-2pt] &R^{n^{''}}\ar[r,  yshift=2pt]\ar[r,  yshift=-2pt] &R^{m^{''}} \ar[r,  yshift=2pt]\ar[r,  yshift=-2pt] &R^{k^{''}} \ar[r, yshift=2pt]\ar[r, yshift=-2pt] & 0
\end{tikzcd}.
\end{equation}
Then the exact splitting (see \ref{diagram}) of all the columns have to be compatible with all the columns as well as all the rows, thus the argument that we had produced for ${\rm Nil}_1(R)$ gets iterated in two directions, i.e. we construct two compatible split exact sequences of columns and rows with the same graded entries. We call this a ``net". So working with the generators in $({B^q})^n$ we iterate the same construction and get an $n$-directional net.\\
Hence, we are able to state the following,
\begin{theorem}
    If $R$ is a Pr{\"u}fer domain then ${\rm Nil}_n(R)$ is trivial for all $n\geq 0.$
\end{theorem}
Thus, we record our final statement in this direction below.
\begin{theorem}
    If $R$ is a Pr{\"u}fer domain then $K_n(R) \cong K_n(R[t])$ for all $n \geq 0.$
\end{theorem}
As an obvious corollary to the theorem would be for valuation rings.
\begin{corollary}
     If $R$ is a Pr{\"u}fer domain then $K_n(R) \cong K_n(R[t])$ for all $n \geq 0.$
\end{corollary}
\begin{remark}
    For a detailed discussion on homotopy invariance of $K$-theory on Pr{\"u}fer domains, please refer to \cite{BS22}.
\end{remark}
As noetherian Pr{\"u}fer domains are Dedekind domains we recover a fundamental theorem of $K$-theory for Dedekind domain that is 
\begin{theorem}
     Let $R$ be a Dedekind domain, then $K_n(R[s]) \cong K_n(R)$, in addition $$K_n(R[s,s^{-1}]) \cong K_n(R) \oplus K_{n-1}(R).$$
\end{theorem}
\begin{proof}
    Since a Dedekind domain is a noetherian Pr{\"u}fer domain, the first part of the theorem is trivially true. The proof of $K_n(R[s,s^{-1}]) \cong K_n(R) \oplus K_{n-1}(R)$ follows from  the proof of \cite{wei1}, Section V, Theorem 6.3.
\end{proof}

 \printbibliography
 \end{document}